\documentclass[11pt]{amsart}
\usepackage{amsopn}
\usepackage{amssymb, amscd}
\usepackage{multirow}

\newcommand{\nc}{\newcommand}

\nc{\fg}{\mathfrak{f} } \nc{\vg}{\mathfrak{v} } \nc{\wg}{\mathfrak{w} }
\nc{\zg}{\mathfrak{z} } \nc{\ngo}{\mathfrak{n} } \nc{\kg}{\mathfrak{k} }
\nc{\mg}{\mathfrak{m} } \nc{\bg}{\mathfrak{b} } \nc{\ggo}{\mathfrak{g} }
\nc{\ggob}{\overline{\mathfrak{g}} } \nc{\sog}{\mathfrak{so} }
\nc{\sug}{\mathfrak{su} } \nc{\spg}{\mathfrak{sp} } \nc{\slg}{\mathfrak{sl} }
\nc{\glg}{\mathfrak{gl} } \nc{\cg}{\mathfrak{c} } \nc{\rg}{\mathfrak{r} }
\nc{\hg}{\mathfrak{h} } \nc{\tg}{\mathfrak{t} } \nc{\ug}{\mathfrak{u} }
\nc{\dg}{\mathfrak{d} } \nc{\ag}{\mathfrak{a} } \nc{\pg}{\mathfrak{p} }
\nc{\sg}{\mathfrak{s} } \nc{\affg}{\mathfrak{aff} } \nc{\qg}{\mathfrak{q} }

\nc{\pca}{\mathcal{P}} \nc{\nca}{\mathcal{N}} \nc{\lca}{\mathcal{L}}
\nc{\oca}{\mathcal{O}} \nc{\mca}{\mathcal{M}} \nc{\tca}{\mathcal{T}}
\nc{\aca}{\mathcal{A}} \nc{\cca}{\mathcal{C}} \nc{\gca}{\mathcal{G}}
\nc{\sca}{\mathcal{S}} \nc{\hca}{\mathcal{H}} \nc{\bca}{\mathcal{B}}
\nc{\dca}{\mathcal{D}} \nc{\val}{\operatorname{val}}

\nc{\vp}{\varphi} \nc{\ddt}{\frac{d}{dt}} \nc{\dds}{\frac{d}{ds}}
\nc{\dpar}{\frac{\partial}{\partial t}} \nc{\im}{\mathrm{i}}

\nc{\SO}{\mathrm{SO}} \nc{\Spe}{\mathrm{Sp}} \nc{\Sl}{\mathrm{SL}}
\nc{\SU}{\mathrm{SU}} \nc{\Or}{\mathrm{O}} \nc{\U}{\mathrm{U}} \nc{\Gl}{\mathrm{GL}}
\nc{\Se}{\mathrm{S}} \nc{\Cl}{\mathrm{Cl}} \nc{\Spein}{\mathrm{Spin}}
\nc{\Pin}{\mathrm{Pin}} \nc{\G}{\mathrm{GL}_n(\RR)} \nc{\g}{\mathfrak{gl}_n(\RR)}

\nc{\RR}{{\Bbb R}} \nc{\HH}{{\Bbb H}} \nc{\CC}{{\Bbb C}} \nc{\ZZ}{{\Bbb Z}}
\nc{\FF}{{\Bbb F}} \nc{\NN}{{\Bbb N}} \nc{\QQ}{{\Bbb Q}} \nc{\PP}{{\Bbb P}} \nc{\OO}{{\Bbb O}}

\nc{\vs}{\vspace{.2cm}} \nc{\vsp}{\vspace{1cm}} \nc{\ip}{\langle\cdot,\cdot\rangle}
\nc{\ipp}{(\cdot,\cdot)} \nc{\la}{\langle} \nc{\ra}{\rangle} \nc{\unm}{\frac{1}{2}}
\nc{\unc}{\frac{1}{4}} \nc{\und}{\frac{1}{16}} \nc{\no}{\vs\noindent}
\nc{\lam}{\Lambda^2(\RR^n)^*\otimes\RR^n} \nc{\tangz}{{\rm T}^{\rm Zar}}
\nc{\nor}{{\sf n}}  \nc{\mum}{/\!\!/} \nc{\kir}{/\!\!/\!\!/}
\nc{\Ri}{\tfrac{4\Ric_{\mu}}{||\mu||^2}} \nc{\ds}{\displaystyle}
\nc{\ben}{\begin{enumerate}} \nc{\een}{\end{enumerate}} \nc{\f}{\frac}
\nc{\lb}{[\cdot,\cdot]} \nc{\isn}{\tfrac{1}{||v||^2}}
\nc{\gkp}{(\ggo=\kg\oplus\pg,\ip)} \nc{\ukh}{(\ug=\kg\oplus\hg,\ip)}
\nc{\tgkp}{(\tilde{\ggo}=\kg\oplus\pg,\ip)}
\nc{\wt}{\widetilde}
\nc{\iop}{\mathtt{i}} \nc{\jop}{\mathtt{j}}

\nc{\Hess}{\operatorname{Hess}} \nc{\ad}{\operatorname{ad}}
\nc{\Ad}{\operatorname{Ad}} \nc{\rank}{\operatorname{rank}}
\nc{\Irr}{\operatorname{Irr}} \nc{\End}{\operatorname{End}}
\nc{\Aut}{\operatorname{Aut}} \nc{\Inn}{\operatorname{Inn}}
\nc{\Der}{\operatorname{Der}} \nc{\Ker}{\operatorname{Ker}}
\nc{\Iso}{\operatorname{Iso}} \nc{\Diff}{\operatorname{Diff}}
\nc{\Lie}{\operatorname{L}} \nc{\tr}{\operatorname{tr}} \nc{\dif}{\operatorname{d}}
\nc{\sen}{\operatorname{sen}} \nc{\modu}{\operatorname{mod}}
\nc{\CRic}{\operatorname{PP}} \nc{\Cric}{\operatorname{P}} \nc{\Ricci}{\operatorname{Ric}}
\nc{\sym}{\operatorname{sym}} \nc{\herm}{\operatorname{herm}} \nc{\symac}{\operatorname{sym^{ac}}}
\nc{\symc}{\operatorname{sym^{c}}} \nc{\scalar}{\operatorname{scal}}
\nc{\grad}{\operatorname{grad}} \nc{\ricci}{\operatorname{ric}}
\nc{\Nor}{\operatorname{Norm}}  \nc{\ricc}{\operatorname{Rc^{c}}}
\nc{\Ricc}{\operatorname{Ric^{c}}} \nc{\ricac}{\operatorname{Rc^{ac}}}
\nc{\Ricac}{\operatorname{Ric^{ac}}} \nc{\Riem}{\operatorname{Rm}}
\nc{\riccig}{\operatorname{ric^{\gamma}}} \nc{\Mm}{\operatorname{M}}
\nc{\Bk}{\operatorname{B}} \nc{\En}{\operatorname{E}}
\nc{\Le}{\operatorname{L}} \nc{\tang}{\operatorname{T}}
\nc{\level}{\operatorname{level}} \nc{\rad}{\operatorname{r}}
\nc{\abel}{\operatorname{ab}} \nc{\CH}{\operatorname{CH}}
\nc{\mcc}{\operatorname{mcc}} \nc{\Adj}{\operatorname{Adj}}
\nc{\Order}{\operatorname{O}}  \nc{\inj}{\operatorname{inj}} \nc{\proy}{\operatorname{proy}}
\nc{\vol}{\operatorname{vol}} \nc{\Diag}{\operatorname{Dg}}
\nc{\Spec}{\operatorname{Spec}} \nc{\Ima}{\operatorname{Im}} \nc{\Rea}{\operatorname{Re}}
\nc{\spann}{\operatorname{span}} \nc{\mm}{\operatorname{m}} \nc{\Crit}{\operatorname{Crit}}

\theoremstyle{plain}
\newtheorem{theorem}{Theorem}[section]

\newtheorem{lemma}[theorem]{Lemma}

\theoremstyle{definition}

\theoremstyle{remark}

\title{The Ricci pinching functional on solvmanifolds II}

\author{Jorge Lauret} \author{Cynthia E. Will}

\address{Universidad Nacional de C\'ordoba, FaMAF and CIEM, 5000 C\'ordoba, Argentina}
\email{lauret@famaf.unc.edu.ar}  \email{cwill@famaf.unc.edu.ar}

\thanks{This research was partially supported by grants from FONCYT and SeCyT (Universidad Nacional de C\'ordoba)}

\begin{document}

\maketitle

\begin{abstract}
It is natural to ask whether solvsolitons are global maxima for the Ricci pinching functional 
$F:=\frac{\scalar^2}{|\Ricci|^2}$ on the set of all left-invariant metrics on a given solvable Lie group $S$, as it is to ask whether they are the only global maxima.   A positive answer to both questions was given in a recent paper by the same authors when the Lie algebra $\sg$ of $S$ is either unimodular or has a codimension-one abelian ideal.  In the present paper, we prove that this also holds in the following two more general cases: 1) $\sg$ has a nilradical of codimension-one; 2) the nilradical $\ngo$ of $\sg$ is abelian and the functional $F$ is restricted to the set of metrics such that $\ag\perp\ngo$, where $\sg=\ag\oplus\ngo$ is the orthogonal decomposition with respect to the solvsoliton.  
\end{abstract}


\section{Introduction}

A left-invariant metric on a simply connected solvable Lie group $S$ is called a {\it solvsoliton} when its Ricci operator satisfies 
$$
\Ricci=cI+D, \qquad\mbox{for some} \quad c\in\RR, \quad D\in\Der(\sg), 
$$
where $\sg$ denotes the Lie algebra of $S$ (see \cite{solvsolitons}).  The definition is a neat combination of geometric and algebraic aspects of a Lie group and the following facts explain very well and from different points of view why these metrics are quite distinguished: 

\begin{itemize}
\item {\it Ricci solitons}.  Solvsolitons are all Ricci solitons, they are precisely the left-invariant Ricci solitons such that the Ricci flow evolves by just scaling and pullback by automorphisms (see \cite{solvsolitons} and \cite[Lemma 5.3]{Jbl2}).  Moreover, if S is of real type, then any scalar-curvature normalized Ricci flow solution converges in Cheeger-Gromov topology to a non-flat solvsoliton on a possibly different solvable Lie group, which does not depend on the initial metric (see \cite[Theorem A]{BhmLfn}).

\item {\it Uniqueness}.  On a given $S$, there is at most one solvsoliton up to scaling and pullback by automorphisms of $S$ (see \cite[Section 5]{solvsolitons} and \cite[Corollary 4.3]{BhmLfn}).  

\item {\it Structure}.  If $\ngo$ is the niradical of $\sg$ and $\sg=\ag\oplus\ngo$ is the orthogonal decomposition with respect to a solvsoliton, then $[\ag,\ag]=0$ (which already imposes an algebraic constraint on $\sg$, as the Lie algebra $[e_1,e_3]=e_3$, $[e_2,e_4]=e_4$, $[e_1,e_2]=e_5$  shows) and any $\ad{Y}|_{\ngo}$ must be a semisimple operator, although the strongest and less understood  obstruction is that $\ngo$ has to admit itself a solvsoliton, called a {\it nilsoliton} in the nilpotent case (see \cite[Theorem 4.8]{solvsolitons} and \cite{cruzchica}).  We refer to \cite{Wll,Frn} and the references therein for low-dimensional classification results for solvsolitons.  

\item {\it Maximal symmetry}.  The dimension of the isometry group of a non-flat solvsoliton on $S$ is maximal among all left-invariant metrics on $S$ (see \cite[Corollary C]{BhmLfn}).  A stronger maximality condition holds in the case of a unimodular $S$: the isometry group of a solvsoliton contains all possible isometry groups of left-invariant metrics on $S$ up to conjugation by a diffeomorphism (see \cite[Corollary 1.3]{Jbl3}, and see \cite{GrdJbl} and \cite[Corollary D]{BhmLfn} for the (non-unimodular) Einstein case, where such a diffeomorphism is actually an automorphism of $S$).

\item {\it Ricci-pinched}.  On a given unimodular $S$, solvsolitons are the only global maxima for the Ricci pinching functional 
$$
F:=\frac{\scalar^2}{|\Ricci|^2},
$$ 
restricted to the set of all left-invariant metrics on $S$ (see \cite{einsteinsolv} for nilsolitons and \cite[Theorem 1.2, (v)]{RP} for the general unimodular case).  Note that $F$ is measuring in a sense how far is a metric from being Einstein.  
\end{itemize}

We are concerned in this paper with the last property of solvsolitons above.  Beyond the unimodular case, solvsolitons were proved  in \cite{RP} to be the only global maxima of $F$ for any {\it almost-abelian} $S$ (i.e.\ $\sg$ has a codimension-one abelian ideal).   Our purpose here is to prove that this also holds among the following two much broader classes of solvable Lie groups.   

\begin{theorem}\label{An-thm}
Let $\sg$ be a solvable Lie algebra with nilradical of codimension-one and assume that $S$ admits a solvsoliton $g$.  Then $g$ is a global maximum for the functional $F$ restricted to the set of all left-invariant metrics on $S$.  Moreover, any other global maximum $g'$ is also a solvsoliton (i.e.\ $g'=c\vp^*g$, for some $c>0$ and $\vp\in\Aut(S)$). 
\end{theorem}

\begin{theorem}\label{Ai-thm}
Let $\sg$ be a solvable Lie algebra with abelian nilradical $\ngo$ and assume that $S$ admits a solvsoliton $g$.  If $\sg=\ag\oplus\ngo$ is orthogonal with respect to $g$, then the solvsoliton $g$ is a global maximum for the functional $F$ restricted to the set of all left-invariant metrics on $S$ such that $\ag\perp\ngo$.  Moreover, any other global maximum $g'$ for $F$ on such a set is also a solvsoliton (i.e.\ $g'=c\vp^*g$, for some $c>0$ and $\vp\in\Aut(S)$). 
\end{theorem}

The proofs of these theorems will be given in Sections \ref{An-sec} and \ref{Ai-sec}, respectively.

\section{The functional $F$ on rank-one solvmanifolds}\label{An-sec} 

Let $S$ be a solvable Lie group of dimension $n$ such that the nilradical $\ngo$ of its Lie algebra $\sg$ is non-abelian and has dimension $n-1$.  If we fix a decomposition $\sg=\RR Y\oplus\ngo$, then the Lie bracket of $\sg$ is determined by the pair $(A,\lb_\ngo)$, where $A:=\ad{Y}|_{\ngo}\in\Der(\ngo)$ and $\lb_\ngo$ is the Lie bracket of $\ngo$.  

By fixing in addition an inner-product $\ip$ on $\sg$ such that $Y\perp\ngo$ and $|Y|=1$, each pair $(A,\lb_\ngo)$ is identified with the corresponding solvable Lie group $S_{(A,\lb_\ngo)}$ endowed with the left-invariant metric defined by $\ip$.  The space of such pairs therefore covers, up to isometry, all left-invariant metrics on Lie groups with a codimension-one nilradical.  Indeed, more precisely, the left-invariant metric $\la\overline{h}\cdot,\overline{h}\cdot\ra$ on the Lie group $S_{(A,\lb_\ngo)}$, where, with respect to a fixed orthonormal basis $\{ Y, X_1,\dots,X_{n-1}\}$ of $(\sg,\ip)$,  
$$
\overline{h}:=\left[\begin{array}{c|ccc}
c^{-1} & &0& \\ \hline &\\ 
X&&h& \\&
\end{array}\right], \qquad c\in\RR^*, \quad X\in\ngo,  \quad h\in\Gl(\ngo), 
$$
is isometric to the pair $(ch(A-\ad_\ngo{h^{-1}X})h^{-1},h\cdot\lb_\ngo)$, since $\overline{h}$ is an isomorphism between the corresponding Lie algebras.    

It follows from \cite[(25)]{solvsolitons} that the Ricci curvature of $(A,\lb_\ngo)$ is given by,
\begin{equation}\label{ricAn}
\Ricci = \left[\begin{array}{c|c}
-\tr{S(A)^2} & \ast \\ \hline &\\ 
\ast&\Ricci_{\lb_\ngo}+\unm[A,A^t]-(\tr{A})S(A) \\&
\end{array}\right],  
\end{equation}
where $S(A):=\unm(A+A^t)$ and $\Ricci_{\lb_\ngo}$ denotes de Ricci operator of the nilmanifold $(\lb_\ngo,\ip|_{\ngo\times\ngo})$.  If $|\lb_\ngo|=2$, which can be assumed up to scaling, then $\tr{\Ricci_{\lb_\ngo}}=-1$ and so
\begin{equation}\label{FAn}
F(A,\lb_\ngo) = \frac{\left(\tr{S(A)^2}+(\tr{A})^2+1\right)^2}{\tr{S(A)^2}\left(\tr{S(A)^2}+(\tr{A})^2\right) + |\Ricci_{\lb_\ngo}|^2 + \unc|[A,A^t]|^2 + G(A,\lb_\ngo)},
\end{equation}
for some expression $G(A,\lb_\ngo)\geq 0$ that vanishes if $A$ is normal.  

Let us suppose that $(A,\lb_\ngo)$ is a solvsoliton with $|\lb_\ngo|=2$, i.e.\ $\Ricci+|\Ricci_{\lb_\ngo}|^2I\in\Der(\sg)$, which is equivalent by \cite[Theorem 4.8]{solvsolitons} to
\begin{equation}\label{soliton-cond}
[A,A^t]=0, \qquad \tr{S(A)^2}=|\Ricci_{\lb_\ngo}|^2, \qquad \Ricci_{\lb_\ngo}+|\Ricci_{\lb_\ngo}|^2I\in\Der(\lb_\ngo). 
\end{equation}
Consider $(ch(A-\ad_\ngo{h^{-1}X})h^{-1},h\cdot\lb_\ngo)$, where $c\ne 0$, $X\in\ngo$, $h\in\Gl(\ngo)$, and assume (up to scaling) that $|h\cdot\lb_\ngo|=2$.  In order to prove Theorem \ref{An-thm}, it is therefore enough to show that
\begin{equation}\label{ineq-An}
F(ch(A-\ad_\ngo{h^{-1}X})h^{-1},h\cdot\lb_\ngo)\leq F(A,\lb_\ngo).
\end{equation}

Let $\ngo=\ngo_1\oplus\dots\oplus\ngo_k$ be the orthogonal decomposition such that $[\ngo,\ngo]_\ngo=\ngo_2\oplus\dots\oplus\ngo_k$ and so on with the rest of the descending central series.  Since $A$ is normal, $A^t$ is also a derivation of $\ngo$ and thus relative to this decomposition,  
$$
A = 
\left[\begin{array}{ccc} \begin{array}{l|}A_1\\ \hline\end{array}&0&0\\  0&\ddots&0\\ 0&0&\begin{array}{|r} \hline A_k\end{array} \end{array}\right], \qquad 
\ad_\ngo{h^{-1}X}= 
\left[\begin{array}{ccc} \begin{array}{l|}0\\ \hline\end{array}&0&0\\  \ast&\ddots&0\\ \ast&\ast&\begin{array}{|r} \hline 0\end{array} \end{array}\right],
$$
where the blocks correspond to each $\ngo_i$.   This implies that $A$ belongs to the closure of the conjugation class of $A-\ad_\ngo{h^{-1}X}$ (by conjugating with matrices which are multiples of the identity on each block), and so 
$$
|h(A-\ad_\ngo{h^{-1}X})h^{-1}|\geq |A|,
$$ 
where equality holds if and only if $h(A-\ad_\ngo{h^{-1}X})h^{-1}$ is normal.  Indeed, recall that $A$ is normal and so it is a minimal vector since the moment map for the conjugation $\Gl(\ngo)$-action on $\glg(\ngo)$ is given by $m(A)=[A,A^t]/|A|^2$ (see \cite{RchSld,HnzSchStt}).  On the other hand, 
$$
\tr{S(h(A-\ad_\ngo{h^{-1}X})h^{-1})^2}=\unm|h(A-\ad_\ngo{h^{-1}X})h^{-1}|^2+\unm\tr{A^2},
$$ 
and we also know that 
$|\Ricci_{h\cdot\lb_\ngo}|\geq|\Ricci_{\lb_\ngo}|$, by using that $\lb_\ngo$ is a nilsoliton (see \cite[Section 3.1]{RP}).  It therefore follows from \eqref{FAn} that 
\begin{align}
F(ch(A-\ad_\ngo{h^{-1}X})h^{-1},h\cdot\lb_\ngo) 
\leq& \frac{\left(c^2T+c^2(\tr{A})^2+1\right)^2}{c^4T\left(T+(\tr{A})^2\right) + |\Ricci_{\lb_\ngo}|^2} \label{FAn2} \\ 
\leq& \frac{\left(c^2(x+a)+1\right)^2}{c^4(x+b)(x+a)+x_0+b}=:f(x,c), \notag
\end{align}
where $T:=\tr{S(h(A-\ad_\ngo{h^{-1}X})h^{-1})^2}$ and
$$
x:=\unm|h(A-\ad_\ngo{h^{-1}X})h^{-1}|^2\geq\unm|A|^2=:x_0, \quad b:=\unm\tr{A^2}, \quad a:=b+(\tr{A})^2.
$$
Note that $x_0+b=|\Ricci_{\lb_\ngo}|^2>0$ and that the value of $F$ at the solvsoliton is given by 
$$
F(A,\lb_\ngo) = f(x_0,1) = \frac{x_0+a+1}{x_0+b}. 
$$

\begin{lemma}\label{fAn}
$f(x,c)\leq f(x_0,1)$ for any $x\geq x_0$ and $c\in\RR$, where equality holds if and only if $x=x_0$ and $c=\pm 1$.  
\end{lemma}

\begin{proof}
We first note that if we consider the denominator of $f(x,c)$, given by the parabola
$$
q(x):=c^4(x+b)(x+a)+x_0+b,
$$ 
then $q(x_0)>0$ and $q'(x_0)>0$, from which follows that $q(x)>0$ for any $x\geq x_0$.  It follows that 
\begin{equation}\label{fAn-eq0}
f(x_0,c)\leq f(x_0,1), \qquad \forall c\in\RR,
\end{equation}
where equality holds if and only if $c=\pm 1$, as this is equivalent to 
$$
c^4(x_0+a)^2+1+2c^2(x_0+a) \leq (x_0+a+1)(c^4(x_0+a)+1),
$$
which simplifies to $0\leq (x_0+a)(c^2-1)^2$.  

On the other hand, it is straightforward to show that inequality $f(x,c)\leq f(x_0,1)$ can be written as
$$
p(x) = rx^2+sx+t\geq 0, \quad \forall x\geq x_0, 
$$
where $r=c^4(a-b+1)>0$, $s=c^4\Big(a(a-b+1)-x_0(a-b)+b\Big)-2c^2(x_0+b)$ and 
$$
t=ac^4(x_0(b-a)+b)-2ac^2(x_0+b)+(x_0+a)(x_0+b).  
$$
It follows from \eqref{fAn-eq0} that $p(x_0)\geq 0$, where equality holds if and only if $c=\pm 1$, if $p'(x_0)\geq 0$, then the lemma follows.   One can therefore assume that $p'(x_0)=2rx_0+s<0$, that is,
\begin{equation}\label{fAn-eq}
c^4(a-b)(x_0+a) < -c^4(x_0+a) + c^2(2-c^2)(x_0+b), 
\end{equation}  
which is easily seen to imply that $c^2<1$ by using that $a>b$.  A straightforward computation gives that the discriminant of $p$ equals
$$
s^2-4rt = c^4(a-b)(x_0+a+1)\Big(c^4(a-b)(x_0+a+1) + 4(c^2-1)(x_0+b)\Big), 
$$
and so it is negative since by \eqref{fAn-eq} and the fact that $c^2<1$, the factor on the right is smaller than  
$$
-c^4(x_0+a) + c^2(2-c^2)(x_0+b) + c^4(a-b) + 4(c^2-1)(x_0+b) =  -2(x_0+b)(c^4 -3c^2 +2) <0.
$$
Thus $p$ is always positive, concluding the proof.  

Alternatively, a simple analytic argument using the partial derivatives of $f$ gives that $(x_0,1)$ is a local maximum of the function $f$ on the half plane $\{ (x,c):x\geq x_0\}$ and that the only critical points of $f$ in this region are $\{ (x,c):c^2(x+b)=d\}$, with critical values, 
$$
f(x,c) =  \frac{\left(c^2(x+a)+1\right)^2}{c^2d(x+a)+d} = \frac{c^2(x+a)+1}{d} = \frac{d+\frac{d}{x+b}(a-b)+1}{d} \leq f(x_0,1).  
$$
Since 
$$
\lim_{x\to\infty} f(x,c) = 1 < f(x_0,1), \qquad \lim_{c\to\infty} f(x,c) = \frac{x+a}{x+b} \leq 1+\frac{a-b}{d} < f(x_0,1), 
$$
there are positive numbers $M_1,M_2$ such that the value of $f$ out of the compact region $\{ (x,c):x_0\leq x\leq M_1, \; |c|\leq M_2\}$ is always strictly less than $f(x_0,1)$.  All this implies that $(x_0,\pm 1)$ are actually the only global maxima of $f$ on the half plane $\{ (x,c):x\geq x_0\}$, as desired.  
\end{proof}

Inequality \eqref{ineq-An} therefore follows from \eqref{FAn2} and the inequality given in Lemma \ref{fAn}.  On the other hand, if equality holds in \eqref{FAn2} and Lemma \ref{fAn}, then $h(A-\ad_\ngo{h^{-1}X})h^{-1}$ is normal, $|h(A-\ad_\ngo{h^{-1}X})h^{-1}|=|A|$, $c=\pm 1$ and $|\Ricci_{h\cdot\lb_\ngo}|^2= |\Ricci_{\lb_\ngo}|^2$.  This implies that
\begin{align*}
\tr{S(ch(A-\ad_\ngo{h^{-1}X})h^{-1})^2}=&\unm |h(A-\ad_\ngo{h^{-1}X})h^{-1}|^2+\unm\tr{A^2}=\tr{S(A)^2}\\ =&|\Ricci_{\lb_\ngo}|^2=|\Ricci_{h\cdot\lb_\ngo}|^2,
\end{align*}
and it follows from \cite[Section 3.1]{RP} that $h\cdot\lb_\ngo$ is a nilsoliton.  Thus $h\cdot\lb_\ngo=\tilde{h}\cdot\lb_\ngo$ for some $\tilde{h}\in\Or(\ngo,\ip)$ by the uniqueness of nilsolitons (see \cite[Theorem 3.5]{soliton}) and hence 
$$
\Ricci_{h\cdot\lb_\ngo}+|\Ricci_{h\cdot\lb_\ngo}|^2I = \tilde{h}(\Ricci_{\lb_\ngo}+|\Ricci_{\lb_\ngo}|^2I)\tilde{h}^{-1} \in\Der(\tilde{h}\cdot\lb_\ngo)=\Der(h\cdot\lb_\ngo).  
$$
All this implies that $(ch(A-\ad_\ngo{h^{-1}X})h^{-1},h\cdot\lb_\ngo)$ is a solvoliton (see \eqref{soliton-cond}), concluding the proof of Theorem \ref{An-thm}.

\section{The functional $F$ on solvmanifolds with an abelian nilradical}\label{Ai-sec} 

In this section, we consider a solvable Lie group $S$ of dimension $n$ such that the nilradical $\ngo$ of its Lie algebra $\sg$ is abelian, say with $\dim{\ngo}=n-r$.  After fixing a decomposition $\sg=\ag\oplus\ngo$ and a basis $\{ Y_i\}$ of $\ag$, the Lie bracket of $\sg$ is determined by an $r$-tuple $(A_1,\dots,A_r)$ of linearly independent linear operators of $\ngo$ such that $[A_i,A_j]=0$ for all $i,j$, where $A_i:=\ad{Y_i}|_{\ngo}\in\glg(\ngo)$, and a bilinear map $\lambda:\ag\times\ag\rightarrow\ngo$.  We assume that $S$ admits a solvsoliton, hence $\lambda=0$ for the corresponding orthogonal decomposition $\sg=\ag\oplus\ngo$ (see \cite[Theorem 4.8]{solvsolitons}).  

By fixing an inner-product $\ip$ on $\sg$ such that $\ag\perp\ngo$ and $\{ Y_i\}$ is orthonormal, each $(A_1,\dots,A_r)$ is identified with the corresponding solvable Lie group $S_{(A_1,\dots,A_r)}$ endowed with the left-invariant metric defined by $\ip$.  It is easy to see that any left-invariant metric on the Lie group $S_{(A_1,\dots,A_r)}$ for which $\ag\perp\ngo$ is isometric to some 
\begin{equation}\label{metr}
\left(h_2(h_1^{-1}A_1)h_2^{-1},\dots,h_2(h_1^{-1}A_r)h_2^{-1}\right), \qquad 
h_1\in\Gl(\ag), \quad h_2\in\Gl(\ngo),
\end{equation}
where $h_1^{-1}A_i:= \ad{h_1^{-1}Y_i} = \sum c_{ji}A_j$ if the matrix of $h_1^{-1}$ relative to $\{ Y_i\}$ is $[c_{ij}]$.  

It follows from \cite[(25)]{solvsolitons} that the Ricci curvature of $(A_1,\dots,A_r)$ is given by,
\begin{equation}\label{ricAi}
\Ricci = \left[\begin{array}{c|c}
R & 0 \\ \hline  
0&\unm\sum[A_i,A_i^t]-\sum(\tr{A_i})S(A_i) 
\end{array}\right],  \qquad R_{ij}=-\tr{S(A_i)S(A_j)}.
\end{equation}
Up to isometry, it can always be assumed that $\tr{A_2}=\dots=\tr{A_r}=0$ (i.e.\ $H=(\tr{A_1})Y_1$) by considering in \eqref{metr} $h_2=I$ and a suitable $h_1\in\Or(\ag,\ip)$.  In that case, 
\begin{equation}\label{FAi}
F(A_1,\dots,A_r) = \frac{\left(\sum\tr{S(A_i)^2}+(\tr{A_1})^2\right)^2}
{\sum\left(\tr{S(A_i)S(A_j)}\right)^2 + (\tr{A_1})^2\tr{S(A_1)^2} + \tfrac{1}{4}|\sum [A_i,A_i^t]|^2}.
\end{equation}

Let us suppose that $(A_1,\dots,A_r)$ is a solvsoliton, that is, $\tr{S(A_i)S(A_j)}=\delta_{ij}$ (up to scaling) and $A_i$ is normal for all $i$ (see \cite[Theorem 4.8]{solvsolitons}).  We consider an $r$-uple as in \eqref{metr} and assume (up to isometry and scaling) that $\tr{h_1^{-1}A_1}=\tr{A_1}$ and $\tr{h_1^{-1}A_2}=\dots=\tr{h_1^{-1}A_r}=0$ (i.e.\ $c_{11}=1$ and $c_{12}=\dots=c_{1r}=0$).  Thus what we must show to prove Theorem \ref{Ai-thm} is that
\begin{equation}\label{Ai-ineq}
F\left(h_2(h_1^{-1}A_1)h_2^{-1},\dots,h_2(h_1^{-1}A_r)h_2^{-1}\right)\leq F(A_1,\dots,A_r), 
\end{equation}
for any $h_1\in\Gl(\ag)$ and $h_2\in\Gl(\ngo)$.

By \eqref{FAi}, we have that 
\begin{align}
F\left(h_2(h_1^{-1}A_1)h_2^{-1},\dots,h_2(h_1^{-1}A_r)h_2^{-1}\right) \leq&  \frac{(x_1+\dots+x_r+a)^2}{x_1^2+\dots+x_r^2+ax_1}=:f(x_1,\dots,x_r), \label{FAi2}
\end{align}
where 
$$
x_i:=\tr{S(h_2(h_1^{-1}A_i)h_2^{-1})^2}, \qquad a:=(\tr{A_1})^2.  
$$
Since $h_1^{-1}A_1=A_1+c_{21}A_2+\dots+c_{r1}A_r$ is normal, it follows from \cite[(17)]{RP} that
\begin{align*}
x_1 =& \unm\tr{(h_1^{-1}A_1)^2} + \unm|h_2(h_1^{-1}A_1)h_2^{-1}|^2 \geq  \unm\tr{(h_1^{-1}A_1)^2} + \unm|h_1^{-1}A_1|^2 \\ 
=& \tr{S(h_1^{-1}A_1)^2} = 1 + \sum_{i=2}^r c_{i1}^2 \geq 1.
\end{align*}
Note that the value of $F$ at the solvsoliton is given by 
$$
F(A_1,\dots,A_r)) = f(1,\dots,1) = r+a. 
$$

\begin{lemma}\label{fAi}
$f(x_1,\dots,x_r)\leq f(1,\dots,1)$ for any $x_1\geq 1$, $x_2,\dots,x_r>0$, where equality holds if and only if $x_1=\dots=x_r=1$.  
\end{lemma}

\begin{proof}
An elementary algebraic manipulation gives that the inequality is equivalent to 
$$
0\leq \left((r-1)\sum x_i^2-\sum x_ix_j\right) + a^2(x_1-1) + a\left(\sum x_i^2-2\sum x_i +rx_1\right).  
$$
Since the first term is $\geq 0$ by the Cauchy-Schwartz inequality $(\sum x_i)^2\leq r\sum x_i^2$ and the third one is $\geq a\sum (x_i-1)^2$, one obtains that both the above inequality and the equality condition in the lemma follow.    
\end{proof}

Since $F$ is invariant under all the assumptions made above up to isometry and scaling, inequality \eqref{Ai-ineq} follows from \eqref{FAi2} and Lemma \ref{fAi}.  Moreover, if equality holds, then  $h_2(h_1^{-1}A_i)h_2^{-1}$ is normal for all $i$ and 
$$
\tr{S(h_2(h_1^{-1}A_i)h_2^{-1})S(h_2(h_1^{-1}A_j)h_2^{-1})}=\delta_{ij},
$$  
which implies that $\left(h_2(h_1^{-1}A_1)h_2^{-1},\dots,h_2(h_1^{-1}A_r)h_2^{-1}\right)$ is a solvsoliton, concluding the proof of Theorem \ref{Ai-thm}.

\end{document}